\documentclass[11pt]{amsart}
\usepackage[utf8x]{inputenc}
\usepackage{amsfonts}
\usepackage{amssymb,epsfig}
\usepackage{amsmath}
\usepackage{delarray}
\usepackage{graphicx} 
\usepackage{xcolor}   
\usepackage{url}
\usepackage{tikz}     
\usepackage{lmodern}           
\usepackage{latexsym}
\usepackage[all]{xy}

\newtheorem{theorem}{Theorem}

\newtheorem{lemma}[theorem]{Lemma}

\newtheorem{notation}[theorem]{Notation}
\newtheorem{corollary}[theorem]{Corollary}
\newtheorem{mythr}{Theorem}

\newtheorem{thr}{Theorem}

\newtheorem{example}[theorem]{Exemple}
\newtheorem{remark}[theorem]{Remark}

\newcommand{\C}{\ensuremath{\mathbb C}}

\newcommand{\Q}{\ensuremath{\mathbb Q}}

\newcommand{\N}{\ensuremath{\mathbb N}}

\newcommand{\Ima}{\mathrm{Im}}
\newcommand{\Ker}{\mathrm{Ker}}

\newcommand{\OO}{\mathcal{O}}
\newcommand{\DD}{\mathcal{D}}

\newcommand{\FF}{\mathcal{F}}

\newcommand{\MM}{\mathcal{M}}

\newcommand{\QQ}{\mathcal{Q}}

\newcommand{\tF}{\tilde{\mathcal{F}}}

\newcommand{\dx}{\partial_x}
\newcommand{\dy}{\partial_y}
\newcommand{\ddx}{\frac{\partial}{\partial x}}
\newcommand{\ddy}{\frac{\partial}{\partial y}}

\newcommand{\dxal}{\partial_x\alpha}
\newcommand{\dxbe}{\partial_x\beta}
\newcommand{\dyal}{\partial_y\alpha}
\newcommand{\dybe}{\partial_y\beta}

\newcommand{\ta}{\tilde{a}}
\newcommand{\tb}{\tilde{b}}
\newcommand{\bA}{\bar{A}}

\newcommand{\bp}{\bar{p}}
\newcommand{\bB}{\bar{B}}
\newcommand{\bT}{\bar{T}}
\newcommand{\bU}{\bar{U}}
\newcommand{\bV}{\bar{V}}
\newcommand{\tq}{\tilde{q}}

\newcommand{\epz}{\varepsilon_0}
\newcommand{\epi}{\varepsilon_\infty}
\newcommand{\bx}{\bar{x}}
\newcommand{\by}{\bar{y}}

\begin{document}

\title[Normal forms of topologically quasi-homogeneous foliation]{Normal forms of topologically quasi-homogeneous foliations on $(\C^2,0)$}
\author{{\sc Truong} Hong Minh }
\address{Institut de Mathématiques de Toulouse, UMR5219, Université Toulouse 3}
\email{hong-minh.truong@math.univ-toulouse.fr}

\begin{abstract}{
The aim of this paper is to construct formal normal forms for the class of topologically quasi-homogeneous foliations under generic conditions. Any such normal form is given as the sum of three terms: an initial generic quasi-homogeneous term, a hamiltonian term and a radial term. Moreover, we also show that the number of free coefficients in the hamiltonian part is consistent with the dimension of Mattei's moduli space of unfoldings. 
}\end{abstract}

\date{Version 1, June 2013}
\subjclass{}
\keywords{formal normal form, quasi-homogeneous foliation}%

\maketitle



\section*{Introduction and Preliminaries}


A germ of holomorphic function $f:(\C^2,0)\rightarrow(\C,0)$ is \emph{quasi-homogeneous} if $f$ belongs to the jacobian ideal 
$J(f)=(\frac{\partial f}{\partial x},\frac{\partial f}{\partial y})$. If $f$ is quasi-homogeneous, there exist coordinates $(x,y)$ and positive coprime integers $k,\ell$ such that
$R(f)=d\cdot f$, where $R=kx\ddx + \ell y\ddy$ is the quasi-radial vector field and $d$ is the quasi-homogeneous degree of $f$ \cite{Sai}. In these coordinates, $f$ can be written, up to a multiple of a constant, as
$$f=x^{n_0}y^{n_\infty}\prod_{i=1}^n(y^k-f_ix^\ell)^{n_i},$$
where the multiplicities satisfy $n_0\geq 0, n_\infty\geq 0, n_i>0$ and the coefficients $c_i$ are non vanishing such that $f_i\neq f_j$. A germ of holomorphic function $f$ is called \emph{topologically quasi-homogeneous} if its zero level set is topologically conjugated to the zero level set of a quasi-homogeneous function. We also say that a germ of non-dicritical holomorphic foliations 
$\FF$ is \emph{topologically quasi-homogeneous} if after desingularization by successive blowing-ups, none of singularities of strict transform $\tF$ are saddle-node and the separatrices of $\FF$ is the zero level set of a topologically quasi-homogeneous function. The separatrices of $\FF$ that are conjugated to these curves $\{y^k-c_i x^\ell =0\}$ are called the \emph{cuspidal branches}. The one (if exists) that is conjugated to $\{x^{n_0}=0\}$ or $\{y^{n_\infty} =0\}$ is called the \emph{$y$-axis branch} or \emph{$x$-axis branch} respectively. We call $\FF$ \emph{topologically quasi-homogeneous with axis branches} if it admits both $x$-axis and $y$-axis branches. 

Two germs of $1$-forms $\omega$ and $\omega'$ are called \emph{orbitally equivalent} if there exists a local change of coordinates $\phi$ and an invertible function $u$ such that 
$$u.\phi^*\omega=\omega'.$$
If moreover $D\phi(0)=\mathrm{Id}$ and $u(0)=1$ then the equivalence will be called \emph{strict}.
\begin{thr}
Let $\omega$ be a $1$-form which defines a topologically quasi-homogeneous foliation with axis branches. Under a generic condition, $\omega$ is strictly formally orbitally equivalent to a unique form $\omega_{h,s}$
\begin{equation*}
\omega_{h,s}=\omega_d+d\left(xy h\right)+s(kydx-\ell xdy)
\end{equation*}
where 
\begin{equation*}
\omega_d=c_0xy\prod_{i=1}^n(y^k-c_ix^\ell)\left(\sum_{i=1}^n\lambda_i\frac{d(y^k-c_ix^\ell)}{y^k-c_ix^\ell}+(\ell\lambda_0+\ell-u)\frac{dx}{x}+(k\lambda_\infty+v)\frac{dy}{y}\right),
\end{equation*}
\begin{equation*}
h(x,y)=\sum_{\substack{ki+\ell j\geq k\ell n+1\\ 0\leq i\leq \ell n-1\\ 0\leq j\leq kn-1}} h_{ij}x^iy^j,\; s(x,y)=\sum_{j=0}^{kn-1}s_j(x)x^{\ell n+1+]\frac{1-\ell j}{k}]}y^j,
\end{equation*} 
$s_i(x)$ are formal series on $x$, $\big]\frac{1-\ell j}{k}\big]$ stands for strict integer part of $\frac{1-\ell j}{k}$ defined as $\big]\frac{1-\ell j}{k}\big]<\frac{1-\ell j}{k}\leq \big]\frac{1-\ell j}{k}\big]+1$.
\end{thr}

\section{Desingularization process of quasi-homogeneous functions}
Fix a reduced quasi-homogeneous function $f=x^{\epz}y^{\epi}\prod_{i=1}^n(y^k-f_i x^\ell)$. Let us recall the algorithm of desingularization of $f$ and its atlas of charts.

On the blowing-up of $(\C^2,0)$ endowed with the chart $(x,y)$, we will use the standard charts $(x,\by),(\bx,y)$ together with the transition functions $\bx=\by^{-1}, y=x\by$. The center
of the first chart $(x,\by)$ is denoted by $0$ and the center of the second one is denoted by $\infty$ (figure \ref{fig2.1}). We denote by 
\begin{equation*}
\sigma=\sigma_1\circ\ldots\circ \sigma_p:(\MM,\DD)\rightarrow (\C^2,0) 
\end{equation*}
the desingularization map of $f$
obtained by composition of the blowing-up's $\sigma_i$, $1\leq i\leq p$, and $\DD=\sigma^{-1}(0)$ the exceptional divisor.
\begin{figure}
\begin{center}
\includegraphics[scale=.8]{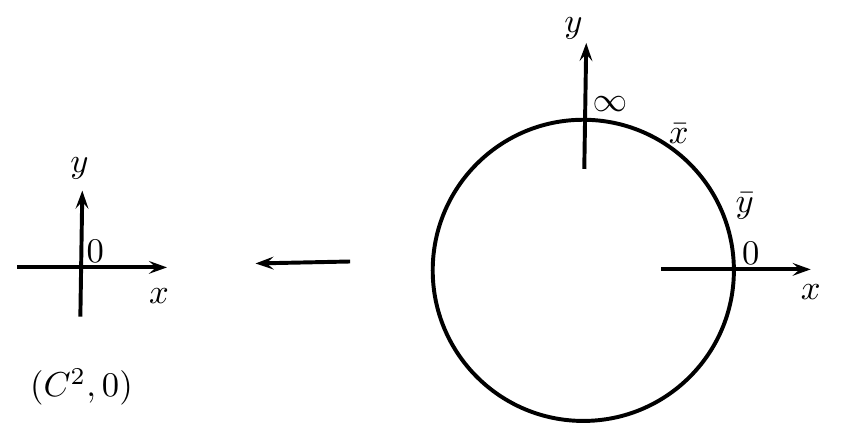}
\caption{Blowing-up at origin of $(\C^2,0)$}\label{fig2.1}
\end{center}
\end{figure}
Let us sketch some properties of $\sigma$. In the desingularization process, we only have to use blowing-up of $0$ or $\infty$. Therefore, the tree of exceptional divisor is a totally ordered sequence of $N$ 
components covered by $N+1$ charts and the map $\sigma$ is monomial in each chart. Before the last blowing-up, all cuspidal branches share the same infinitesimal point. 
After the last blowing-up, they appear on the same component of $\DD$ called the principal component. If $\epz\neq 0$ or $\epi\neq 0$, 
the corresponding strict branches appear on the end components. Let us number the components of $\DD$ and their charts in such a way that $D_1$ corresponds to the strict 
branches which appears if $\epi\neq 0$. So, we obtain $N+1$ chart $(x_i,y_i), i=0,\ldots,N$, such that each component $D_i,i=1\ldots N$ is covered by domains 
$V_{i-1}$ and $V_i$ of the charts $(x_{i-1},y_{i-1})$ around $(D_i,0)$ and $(D_i,\infty)$ (figure \ref{fig2.2}). The change of charts is given by
$$x_{i}=y_{i-1}^{-1}, y_{i}=x_{i-1}y_{i-1}^{e_i}$$
where $-e_i$ is the self intersection number of the component $D_i$. We denote by $c$ the index corresponding to the principal component. Then, the desingularization map $\sigma$ is given
in the chart $(x_c,y_c)$ by $(x,y)=(x_c^{k-v}y_c^k,x_c^{\ell-u}y_c^\ell)$ \cite{Gen-Pau2}, where $u$, $v$ two positive integers such that 
$ku-\ell v=1$ and $u\leq\ell, v\leq k$.
\begin{figure}
\begin{center}
\includegraphics[scale=.8]{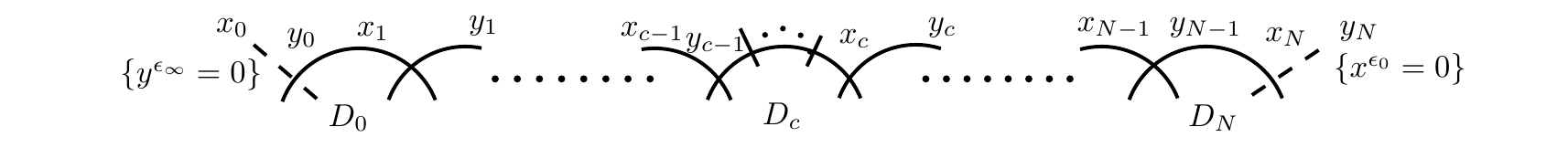}
\caption{Desingularization of $f$}\label{fig2.2}
\end{center}
\end{figure}

\section{Desingularization of topologically quasi-homogeneous functions}
If two germs of holomorphic functions are topologically conjugated, they admit the same dual tree of desingularization.
In particular, their desingularization maps have the same number of blowing-ups but they are not necessarily equal. The following lemma shows that in the case topologically 
quasi-homogeneous they can share the same desingularization map after a local change of coordinates.  
\begin{lemma}\label{lem2.3.1}
If a reduced function $f$ is topologically conjugated to $f$ then there exists a local change of coordinates $\phi$ such that $f\circ\phi$ has the same desingularization 
map as $f$. Moreover, $\phi$ can be chosen such that $x^{\epz}y^{\epi}$ divides $f\circ\phi$.  
\end{lemma}
\begin{proof} Without loss of generality, we can assume that $k\ge\ell$ and denote by $m$ the integer part of $\frac{k}{\ell}$. Denote by
$$\sigma'=\sigma'_1\circ\ldots\circ \sigma'_h:(\MM',\DD')\rightarrow (\C^2,0)$$ 
the desingularization maps of $f$ and
$$\sigma^i=\sigma_1\circ\ldots\circ \sigma_i\;\,  \text{and}\;\, \sigma'^i=\sigma'_1\circ\ldots\circ \sigma'_i$$
for $i=1,\ldots,h$. It is easy to see that if $\sigma^{m+1}=\sigma'^{m+1}$ then $\sigma=\sigma'$.

Let us first consider the case $\epi=0$. If $\epz=1$, we will show that any  diffeomorphism $\phi$ that sends $\{x=0\}$ to the 
$y$-axis branch $L_y$ of $f$ is the desired diffeomorphism. Indeed, the center of the blowing-up $\sigma'_i$, with $2\leq i\leq m+1$, is the intersection point of the transform 
$(\sigma'^{i-1})^*(L_y)$ and the divisor $(\sigma'^{i-1})^{-1}(0)$. So after the local change of coordinate $\phi$, the $m+1$ first 
blowing-ups of the desingularization map of $f\circ\phi$ and $f$ are the same. Consequently, $f\circ\phi$ and $f$ have the same desingularization map. In the case $\epz=0$, after $m$ first blowing-ups, all the strict transforms of principal branches of $f$ share the same intersection point $z$ with divisor $(\sigma'^{m})^{-1}(0)$.  We take a smooth curve $\tilde{L}$ transverse to the divisor at $z$ and denote by $L$ the image by $\sigma^m$ of $\tilde{L}$.  Then $L$ is a germ of smooth curve of $(\C^2,0)$. With the same reason as above, any  diffeomorphism that sends $\{x=0\}$ to $L$  is the desired diffeomorphism.

Now we consider the case $\epi=1$. Using the same argument as above, if $\epz=1$ then $\phi$ is a diffeomorphism that sends $x$-axis branch $L_x$ and $y$-axis branch $L_y$ to $\{y=0\}$ and $\{x=0\}$ respectively. In the case $\epz=0$, we define $L$ as above then the desired diffeomorphism is the one that sends $L_x$ and $L$ to $\{y=0\}$ and $\{x=0\}$ respectively.
\end{proof}
\begin{figure}
\begin{center}
\includegraphics[scale=1]{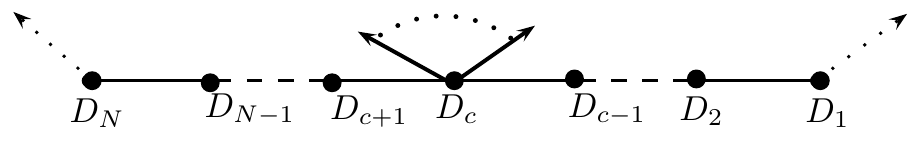}
\caption{Dual tree of topological quasi-homogeneous functions}
\end{center}
\end{figure}

\section{The criteria of topologically quasi-homogeneous foliation}
\begin{notation}
Denote by $\QQ(f)$ the set of all germs of topologically quasi-homogeneous 1-forms in $(\C^2,0)$ whose separatrices are topologically conjugated to $f$ satisfying they admit $\sigma:(\MM,\DD)\rightarrow(\C^2,0)$ as their desingularization map and $x^{\epz}y^{\epi}$ as their invariant curves.
\end{notation}
By the lemma \ref{lem2.3.1}, instead of considering the class of all topologically quasi-homogeneous foliation we can restrict our attention to the subset $\QQ(f)$. 
We will denote by $\OO(k,\ell,d)$ the set of germs of holomorphic functions $q$ satisfying
\begin{equation*}
 q=q_d+q_{d+1}+q_{d+2}\ldots = q_d+ h.o.t.
\end{equation*}
 where $q_d\not\equiv 0$, $q_m$ is $(k,\ell)$-quasi-homogeneous polynomial of degree $m$  and  ``h.o.t.'' stands for higher order term. In what follows, denote by $d=nk\ell+k\epz+\ell\varepsilon_\infty$ the quasi-homogeneous degree of $f$.
\begin{lemma}\label{lem2.4.1}
If a germ of $1$-form $\omega=adx+bdy$ is an element of $\QQ(f)$ then
\begin{itemize}
\item[(i)] $y^{\epi}$ divides $a$, $x^{\epz}$ divides $b$.
\item[(ii)]Let $q:=kxa+\ell yb$, $p:=(k-v)xa+(\ell-u)yb$. Then, $q,p$ are in $\OO(k,\ell,d)$ and
 \begin{equation}\label{pro2-2}
 q_d=c_0x^{\epsilon_0}y^{\varepsilon_\infty}\prod^n_{i=1}(y^k-c_ix^\ell),\,\,c_0\neq0,
\end{equation}
such that all the complex numbers $c_i$, $i\in\{1,\ldots n\}$ are non-zero, different from each other. Moreover, $\mathrm{gcd}(\frac{q_d}{x^\epz y^\epi},\frac{p_d}{x^\epz y^\epi})=1$.
\end{itemize}
\end{lemma}
\begin{proof} It is obvious that $(i)$ is always satisfied. For the convenience of proving by induction on the number of blowing-ups,  we will replace $\QQ(f)$ by $\QQ(k,\ell,n,\epz, \epi)$. For $k=\ell=1$, in the blowing-up coordinates $(\bx,y)$ 
\begin{align*}
\sigma^*\omega(\bx,y)&=ya(\bx y, y)d\bx+(\bx a(\bx y,y)+b(\bx y,y))dy\\
&=\frac{1}{\bx}p(\bx y,y)d\bx+\frac{1}{y}q(\bx y,y)dy\\
&=y^{d-1}\left(\left(\frac{y}{\bx}p_d(\bx,1)+ y^2(\ldots)\right)d\bx+\left(q_d(\bx,1)+ y(\ldots)\right)dy\right).
\end{align*}
Because after desingularization all the singularities are not saddle-node, the roots of $q_d(\bx,1)$ are distinct and none of them are a root of $q_d(\bx,1)$ . So we have
\begin{equation*}
q_d(x,y)=c_0x^{\epz}y^{\epi}\prod_{i=1}^n(y-c_ix)\;\,\text{and}\,\;\mathrm{gcd}(\frac{q_d}{x^\epz y^\epi},\frac{p_d}{x^\epz y^\epi})=1.
\end{equation*}

In general, we can assume without loss of generality that $k>\ell$. Let $\sigma_1$ be the standard blowing-up at the origin. By \cite{Cam-Lin-Sad}, the multiplicity of $\sigma_1^*\omega$ 
equals to the multiplicity of $f\circ \sigma_1$ minus $1$. Hence, $\sigma_1^*\omega$ can be written in the blowing-up coordinates $(\bx,y)$ as follows 
\begin{equation*}
 \sigma_1^*\omega(\bx,y)=ya(\bx y, y)d\bx+(\bx a(\bx y,y)+b(\bx y,y))dy=y^{\epz+\epi+n\ell-1}\omega'(\bx,y).
\end{equation*}
If $\omega$ is in $\QQ(k,\ell,n,\epz, \epi)$ then  $\omega'$ is in $\QQ(k-\ell,\ell,n,\epz,1)$. 
Using the induction hypothesis for $\omega'$, we obtain
\begin{align}\label{pro2-4}
(k-\ell)\bx(ya(\bx y,y))+\ell y(\bx a(\bx y,y)+b(\bx y,y))\nonumber\\
=y^{\epz+\epi+n\ell -1}(c'_0 \bx^{\epz} y\prod^n_{i=1}(y^{k-\ell}-c'_i \bx^{\ell})) + h.o.t. 
\end{align}
where the numbers $c'_i$, $i=1,\ldots,n$, are non-zero, different from each other. Replace $\bx y$ by $x$, \eqref{pro2-4} is equivalent to
$$kxa(x,y)+\ell y b(x,y)=c'_0x^{\epz} y^{\epi}\prod^n_{i=1}(y^k-c'_i x^{l}) + h.o.t. $$
Moreover, we have 
\begin{align*}
y^{\epz+\epi+n\ell-1}p'(\bx,y)&=(k-\ell-v+u)\bx(ya(\bx y,y))+(\ell-u)y(\bx a(\bx y,y)+b(\bx y,y))\\
&=(k-v)\bx ya(\bx y,y)+ (\ell-u)yb(\bx y,y).
\end{align*}
Replace $\bx y$ by $x$ and use the induction hypothesis for $p'$ and $q'$, we obtains $\mathrm{gcd}(\frac{q_d}{x^\epz y^\epi},\frac{p_d}{x^\epz y^\epi})=1$. 
\end{proof}
\begin{remark}
The conditions (i) and (ii) in Lemma \ref{lem2.4.1} are not sufficient for determining the elements in $\QQ(f)$. In fact, a $1$-form $\omega$ satisfies these conditions if and only if the foliation $\FF$ defined by $\omega$ satisfies:
\begin{itemize}
\item[(i)] $\{x^{\epz} y^{\epi}=0\}$ is a invariant curve of $\FF$,
\item[(ii)] Let $\sigma$ be the desinglarization map of $f$. After pull-back bay $\sigma$, except the corners, the strict transform $\sigma^*(\FF)$ has $n$ singularities on principle component $D_c$. Moreover at each singularity $\sigma^*(\FF)$ is defined by a $1$-form whose linear part is 
$$\lambda y dx+xdy, \;\text{where}\;\, \lambda\neq 0.$$
\end{itemize}
So for obtaining $\omega\in\QQ(f)$ we need the condition that all the Camacho-Sad indices of all singularities of $\sigma^*(\FF)$ are not in $\Q_{>0}$.  
\end{remark}
\begin{remark}\label{rem3}
An element $\omega\in\QQ(f)$ can be written as
  \begin{equation}\label{pro2-1}
   \omega=\omega_d+\omega_{d+1}+\omega_{d+2}+\ldots
  \end{equation}
  where $\omega_m=a_{m-k}dx+b_{m-\ell}dy$, $a_{m-k}$, $b_{m-\ell}$ are $(k,\ell)$-quasi-homogeneous polynomials of degrees $m-k$, $m-\ell$ respectively and $\omega_d\not\equiv 0$.
\end{remark}
\begin{remark}
The points $z_i=(\frac{1}{c_i},0)$ $(1\leq i\leq n)$ in the coordinates $(x_c,y_c)$ stand for the intersections of strict transforms of separatrices of $\omega$ with the principle component of divisor.
\end{remark}
\section{Logarithmic representation of the initial part}
Consider $\omega\in\QQ(f)$ having the presentation as in \eqref{pro2-1}. Denote by $\lambda_i$ the Camacho-Sad indices of strict transform foliation $\tF=\sigma^*\FF$ defined by $\sigma^*w=\hat{a}dx+\hat{b}dy$ and the principle 
component $\DD_c$. It means that
\begin{equation*}
\lambda_i=i_{z_i}(\tF,\DD_c)=-\mathrm{Res}_{z_i}\ddy\left(\frac{\hat{a}}{\hat{b}}\right)(x,0). 
\end{equation*}
We also denote by $\lambda_0$ and $\lambda_\infty$ the indices of $\tF$ and $\DD_c$ at $z_0=(x_c=0,y_c=0)$ and $z_\infty=(x_c=\infty,y_c=0)$ respectively. Then by \cite{Cam-Sad} we the the relation: 
\begin{equation}\label{2.5.1}
\sum_{i=1}^n\lambda_i+\lambda_0+\lambda_\infty=-1
\end{equation}
and the projective holonomy $h_i$ of $\tF$ at the point $z_i$ satisfying
$$h_i'(0)=exp(2\pi i\lambda_i)$$
Actually, $(\lambda_i)$ only depends on the quasi-homogeneous part $\omega_d$. Moreover $\omega_d$ is completely determined by $(\lambda_i)$ and $(c_i)$ as in the following Lemma:
\begin{lemma}\label{lem2.5.1} With the notation as above, we have
\begin{equation*} 
\omega_d=-q_d\left(\sum_{i=1}^n\lambda_i\frac{d(y^k-c_ix^\ell)}{y^k-c_ix^\ell}+\epz(\ell(\lambda_0+1)-u)\frac{dx}{x}+\epi(k\lambda_\infty+v)\frac{dy}{y}\right),
\end{equation*}
where $$\sum_{i=1}^n\lambda_i+\epz(\lambda_0+\frac{l-u}{l})+\epi(\lambda_\infty+\frac{v}{k})+\frac{1}{k\ell}=0.$$
\end{lemma}
\begin{proof}
When $\epz=0$ (resp. $\epi=0$), the value of $\lambda_0$ (resp. $\lambda_\infty$) is totally determined by the couple $(k,\ell)$. Therefore, we can calculate the value of $\lambda_0$ (resp. $\lambda_\infty$) when $\epz=0$ (resp. $\epi=0$) by considering the $1$-form $d(y^k-x^{\ell})$. Since $(x,y)=(x_c^{k-v}y_c^k,x_c^{\ell-u}y_c^\ell)=(x_{c-1}^{k}y_{c-1}^v,x_{c-1}^{\ell}y_{c-1}^u)$, we have
\begin{align*}
d(y^k-&x^\ell)\circ\sigma(x_c,y_c)\\
&=x_c^{k\ell-k u-1}y_c^{k\ell-1}\Big(y_c\big(k\ell-ku-(k\ell-ku+1)x_c\big)dx_c+k\ell x_c(1-x_c)dy_c\Big),\\
d(y^k-&x^\ell)\circ\sigma(x_{c-1},y_{c-1})\\
&=x_{c-1}^{k\ell-1} y_{c-1}^{\ell v-1}\Big(k\ell y_{c-1}(y_{c-1}-1)dx_{c-1}+\ell v x_{c-1}(y_{c-1}-1)dy_{c-1}\Big).
\end{align*}
It implies that 
\begin{equation}\label{2.5.2}
\lambda_0=-\frac{v}{k}\,\;\text{when}\; \epz=0\;\text{and}\; \lambda_\infty=-\frac{\ell-u}{\ell}\;\,\text{when}\;\epi=0.
\end{equation}
Now in the coordinates $(x_c,y_c)$, we get
\begin{equation*}
\sigma^*\omega_d=\frac{1}{x_c}p_d(x_c^{k-v}y_c^k,x_c^{\ell-u}y_c^\ell)dx_c+\frac{1}{y_c}q_d(x_c^{k-v}y_c^k,x_c^{\ell-u}y_c^\ell)dy_c,
\end{equation*}
where $p_d=(k-v)xa_{d-k}+(\ell-u)yb_{d-\ell}$, $q_d=kxa_{d-k}+\ell yb_{d-\ell}$. By Lemma \ref{lem2.4.1}, we get
\begin{equation*}
\sigma^*\omega_d=x_c^{e-1}y_c^{d-1}(y_c\bp(x_c)dx_c+c_0x_c\prod_{i=1}^n(1-c_ix_c)dy_c), 
\end{equation*}
where $\bp$ is a polynomial of degree $n$ satisfying $\bp(x_c^\ell)=\frac{1}{x^{\epz}}p(x_c,1)$ and $e=nk(\ell-u)+\epz(k-v)+\epi(\ell-u)$. The definition of $\lambda_i$ leads to
\begin{align*}
\bp(\frac{1}{c_i})&=c_0\lambda_i\prod_{\substack{j=1\\ j\neq i}}^n\left(1-\frac{c_j}{c_i}\right),\; i=1,\ldots,n,\\
\bp(0)&=-c_0\lambda_0.
\end{align*}
Thanks to the formula of Lagrange polynomial, we get
\begin{equation*}
\bp(x_c)=c_0\left(\sum_{i=1}^n\lambda_ic_ix_c\prod_{\substack{j=1\\ j\neq i}}(1-c_jx_c)-\lambda_0\prod_{j=1}^n(1-c_jx_c)\right).
\end{equation*}
It implies that
\begin{align*}
p_d(x,y)&=c_0x^{\epz}y^{\epi}\left(\sum_{i=1}^n\lambda_ic_i x^\ell\prod_{\substack{j=1\\ j\neq i}}(y^k-c_jx^\ell)-\lambda_0\prod_{j=1}^n(y^k-c_jx^\ell)\right)\\
&=q_d\left(\sum_{i=1}^n\frac{\lambda_ic_ix^\ell}{y^k-c_ix^\ell}-\lambda_0\right).
\end{align*}
Consequently, we have
\begin{equation*}\label{eq2.5.3}
a_{d-k}(x,y)=\frac{\ell p_d-(\ell-u)q_d}{x}=\frac{q_d}{x}\left(\sum_{i=1}^n\lambda_i\frac{\ell c_ix^\ell}{y^k-c_ix^\ell}+(u-\ell(\lambda_0+1))\right).
\end{equation*}
\begin{equation*}\label{eq2.5.4}
b_{d-l}(x,y)=\frac{(k-v)q_d-kp_d}{y}=-\frac{q_d}{y}\left(\sum_{i=1}^n\lambda_i\frac{ky^k}{y^k-c_ix^\ell}+(v+k\lambda_\infty)\right).
\end{equation*}
Because \eqref{2.5.2}, we can  
replace $\lambda_0$ by $\epz\lambda_0+(\frac{u}{\ell}-1)(1-\epz)$ and $\lambda_\infty$ by $\epi\lambda_\infty-\frac{v}{k}(1-\epi)$. Therefore 
$u-\ell(\lambda_0+1)$ becomes $\epsilon_0(u-\ell(\lambda_0+1))$, $v+k\lambda_\infty$ becomes $\epi(v+k\lambda_\infty)$ and the relation \eqref{2.5.1} becomes 
$$\sum_{i=1}^n\lambda_i+\epz(\lambda_0+\frac{l-u}{l})+\epi(\lambda_\infty+\frac{v}{k})+\frac{1}{k\ell}=0.$$
\end{proof}

\section{Decomposition of topologically quasi-homogeneous foliations}
\begin{lemma}\label{lem2.1.7}
Let $\omega$ be a germ of $1$-form in $(\C^2,0)$. Then there exist unique holomorphic functions $h$ and $s$ such that
\begin{equation}\label{2.1.8}
\omega=dh+s(\ell ydx-kxdy).
\end{equation}
\end{lemma}
\begin{proof}
Suppose that $\omega=adx+bdy$. Let $R=kx\frac{\partial}{\partial x}+\ell y kx\frac{\partial}{\partial y}$ the quasi-radial vector field. Denote by $q=\omega(R)=kxa+\ell yb$. Suppose that there exit $h$ and $s$ satisfying $\eqref{2.1.8}$. We have
\begin{align*}
q=\omega(R)=R(h).
\end{align*}
It implies that 
\begin{equation}\label{2.1.9}
h=\sum_{i=0}^\infty\frac{q_i}{i}.
\end{equation}
where $q=q_0+q_1+q_2+\ldots$ is the decomposition of $q$ into the $(k,\ell)$ quasi-homogeneous polynomials. This proves the uniqueness part. 

Now assume that $h$ is defined as \eqref{2.1.9}. Decompose $a$, $b$, $h$ into the $(k,\ell)$ quasi-homogeneous polynomials, we have
\begin{align*}
a_{i+\ell}-\dx h_{i+k+\ell}&=a_{i+\ell}-\frac{\dx(kxa_{i+\ell}+\ell yb_{i+k})}{i+k+\ell}\\
&=\frac{(i+\ell)a_{i+\ell}-(kx\dx a_{i+\ell}+\ell y\dx b_{i+k})}{i+k+\ell}=\frac{\ell y(\dy a_{i+\ell}+\dx b_{i+k})}{i+k+\ell},\\
b_{i+k}-\dy h_{i+k+\ell}&=b_{i+k}-\frac{\dy(kxa_{i+\ell}+\ell yb_{i+k})}{i+k+\ell}\\
&=\frac{(i+k)b_{i+k}-(kx\dx a_{i+\ell}+\ell y\dx b_{i+k})}{i+k+\ell}=-\frac{k x(\dy a_{i+\ell}+\dx b_{i+k})}{i+k+\ell}.
\end{align*}
This implies the existence of $s$, which is defined by
$$s_i=\frac{\dy a_{i+\ell}+\dx b_{i+k}}{i+k+\ell}.$$
\end{proof}
Using Lemma \ref{lem2.1.7} for the elements in $\QQ(f)$ we obtains:
\begin{corollary}\label{cor1.1.9}
For each $\omega\in\QQ(f)$, there exist unique holomorphic functions $h$ and $s$ such that
\begin{equation}\label{2.6.1}
\omega=\omega_d+dh+s(\ell y dx-kxdy)
\end{equation}
where $\omega_R$ is the quasi-radial form $lydx-kxdy$. Moreover, we have
\begin{equation}\label{eqpro42}
h(x,y)=\sum_{i=1}^\infty\frac{kxa_{d+i-k}+lyb_{d+i-\ell}}{d+i}, s(x,y)=\sum_{i=1}^\infty\frac{\partial_ya_{d+i-k}-\partial_xb_{d+i-\ell}}{d+i}. 
\end{equation}
\end{corollary}
\section{Formal normal forms of topologically quasi-homogeneous foliations}
In this section, we only consider the case $\epz=\epi=1$. Then $f=xy\prod_{i=1}^n(y^k-f_i x^\ell )$ the homogeneous polynomial of degree $d=k\ell n+k+\ell$.

The process of normalization is follows: Let $\FF$ be a topologically quasi-homogeneous foliations. By Lemma \ref{lem2.3.1}, we can assume that $\FF$ is defined by $\omega\in\QQ(f)$. Decompose $\omega$ as in \eqref{2.6.1}. Then the process is divided into two steps. Firstly, we apply consecutively the diffeomorphisms and the unit multiplications to simplify the hamiltonian part $h$ degree by degree. After that, by using the diffeomorphisms and the unit functions  that do not change the term $h$ we will normalize the function $s$.

Denote by  
\begin{equation*}
\QQ^d(f)=\{\omega_d=a_{d - k}dx+b_{d - \ell}dy:\omega_d\in\QQ(f)\}.
\end{equation*}     
For each $\omega_d\in\QQ^d(f)$, we also denote by $\QQ(\omega_d)$ the subset of $\QQ(f)$ containing the $1$-forms admitting $\omega_d$ as their initial part: 
\begin{equation*}
\QQ(\omega_d)=\{\omega'\in\QQ(f):\omega'=\omega_d+\omega'_{d+1}+\omega'_{d+2}\ldots\}.
\end{equation*}
Now let $\omega(x,y)=a(x,y)dx+b(x,y)dy\in\QQ(f)$. For each integer $m\geq 1$, we consider the local change of coordinates 
$\phi(x,y)=\left(x+\alpha(x,y),y+\beta(x,y)\right)$ and the unity function 
$u(x,y)=1+\delta(x,y)$ where $\alpha,\beta,\delta$ are $(k,\ell)$-quasi-homogeneous polynomials of degrees $m+k,m+\ell,m$ respectively. 
The local change of coordinates $\phi$ and the multiplication by $u$ take the form $\omega$ into the form $$\tilde{\omega}=u.\phi^*\omega=\ta dx+\tb dy.$$ 
\begin{lemma}\label{lem2.7.1}
Denote by $\Delta q=\tq-q$, $\Delta a=\ta-a$, $\Delta b=\tb-b$ we have 
$\Delta q_{d+m'}=\Delta a_{d+m'}=\Delta b_{d+m'}=0$ for all $m'<m$ and
\begin{equation}\label{eqlem61}
\Delta q_{d+m}= AU+BV 
\end{equation}
\begin{equation}\label{eqlem62}
 \Delta b_{d+m-\ell}=W+\frac{1}{d+m}\left(mb_{d-\ell}\delta-q_d\partial_y\delta\right),
\end{equation}
where $A=ma_{d-k}+\partial_xq_d, B=mb_{d-\ell}+\partial_yq_d,U=\alpha+\frac{k}{d+m}x\delta, V=\beta+\frac{l}{d+m}y\delta$ and
$W=\dy(a_{d-k}U+b_{d-\ell}V)-(\dy a_{d-k}-\dx b_{d-\ell})U$.
\end{lemma}
\begin{proof}
We have
$$\phi^*\omega=\left(\left(1+\dxal\right).a\circ\phi+ \dxbe. b\circ\phi\right)dx+\left( \dyal . a\circ\phi+\left(1+\dybe\right).b\circ\phi\right) dy.$$
It implies that
\begin{align}
\tq&=u\left((kx+kx.\dxal+\ell y.\dyal) a\circ\phi+(\ell y + kx.\dxbe+\ell y.\dybe)b\circ\beta\right)\nonumber\\
&=u\left((kx+(k+m)\alpha)a\circ\phi+(\ell y+(\ell+m)\beta)b\circ\phi\right).\nonumber
\end{align}
We also have
\begin{align}
a\circ\phi-a&=\sum_{ki+\ell j\ge d-k}a_{ij}(x+\alpha)^i(y+\beta)^j- \sum_{ki+\ell j\ge d-k}a_{ij}x^i y^j \nonumber\\
&=\sum_{ki+\ell j\ge d-k}a_{ij}i x^{i-1}\alpha y^j + \sum_{ki+\ell j\ge d-k}a_{ij}x^{i} j y^{j-1}\beta + h.o.t. \nonumber\\
&=\alpha\dx a+\beta\dy a+h.o.t.,\label{14}\\
b\circ\phi-b&=\sum_{ki+\ell j\ge d-\ell}b_{ij}(x+\alpha)^i(y+\beta)^j- \sum_{ki+\ell j\ge d-\ell}b_{ij}x^i y^j\nonumber\\
&=\sum_{ki+\ell j\ge d-\ell}b_{ij}i x^{i-1}\alpha y^j + \sum_{ki+\ell j\ge d-\ell}b_{ij}x^{i} j y^{j-1}\beta + h.o.t. \nonumber\\
&=\alpha\dx b+\beta\dy b+h.o.t..\label{15}
\end{align}
It follows that $\Delta q_{d+m'}=0$ for all $0\le m'<m$ and
\begin{align*}
\Delta_{q+m}&=kx(\alpha\dx a+\beta\dy a)+ \ell y(\alpha\dx b+\beta\dy b) + (k+m)\alpha a + (\ell + m)\beta b + \delta q_d \nonumber \\
&=((k+m)a+kx \dx a+\ell y\dx b )\alpha + ((\ell +m)b+ kx\dy a +\ell y \dy b)\beta +\delta q_d \nonumber \\
&=\left(ma_{d-k}+\dx q_d\right)\alpha + \left(mb_{d-\ell}+\dy q_d\right)\beta+\delta q_d. 
\end{align*}
Substituting $$q_d=\frac{k}{d+m}(ma_{d-k}+\dx q_d)+\frac{\ell}{d+m}(mb_{d-\ell}+\dy q_d)$$ we get \eqref{eqlem61}.
Using again \eqref{14} and \eqref{15}, we obtain that $\Delta b_{d-l+m'}=0$ for all $0\le m'<m$ and
\begin{equation}
\Delta b_{d-l+m}=\alpha\dx b_{d-\ell}+\beta \dy b_{d-\ell}+\dy\alpha a_{d-k}+\dy\beta b_{d-\ell}+\delta b_{d-\ell}.\label{16}
\end{equation}
Substituting $\alpha=U-\frac{k}{d+m}x\delta$ and $\beta=V-\frac{\ell}{d+m}y\delta$ into \eqref{16}, we get
\begin{equation*}
\Delta b_{d-l+m}=\dx b_{d-\ell}U+\dy b_{d-\ell}V+a_{d-k}\dy U+b_{d-\ell}\dy V+\frac{1}{d+m}\left(mb_{d-\ell}\delta-q_d\dy\delta\right).
\end{equation*}
It implies \eqref{eqlem62} by using the fact that
\begin{equation*}
 \dx b_{d-\ell}U+\dy b_{d-\ell}V+a_{d-k}\dy U+b_{d-\ell}\dy V=\dy(a_{d-k}U+b_{d-\ell}V)-(\dy a_{d-k}-\dx b_{d-\ell})U.
\end{equation*}
\end{proof}

Denote by $[a[$ the usual integer part $a$: $[a[\leq a < [a[+1$, and $]a]$ the strict integer  part of $a$ defined by $]a]<a\leq ]a]+1$. Then the number of integer points in a closed interval $[a,b]$ is given by $[b[-]a]$. 
\begin{lemma}\label{lem6}Let $e_m$ be the cardinality of the set  $\{(i,j)\in\N^2:ki+\ell j=m\}$. Then  $e_m=[\frac{mu}{\ell}[-]\frac{mv}{k}]$. 
\end{lemma}
\begin{proof}
Denote by $e'_m=[\frac{mu}{\ell}[-]\frac{mv}{k}]$ the number of integer points in the closed interval $[\frac{mu}{\ell},\frac{mv}{k}]$.
For each integer $c$ in $[\frac{mu}{\ell},\frac{mv}{k}]$. Let's put $i=mu-c\ell$, $j=ck-mv$, then
$$ki+\ell j=kmu-ck\ell+ck\ell-\ell mv=m(ku-\ell v)=m.$$
So $e'_m\leq e_m$. Now, if there exit two positive integers $i$, $j$ such that $ki+\ell j=m$ then
\begin{align*}
mu&=kui+\ell uj=(\ell v+1)i+\ell uj=\ell(vi+uj)+i,\\
mv&=kvi+\ell vj=kvi+(ku-1)j=k(vi+uj)-j.
\end{align*}
Therefore, $$vi+uj=\frac{mu-i}{\ell}=\frac{mv+j}{k}\in [\frac{mu}{\ell},\frac{mv}{k}].$$
It implies that $e_m\leq e'_m$. So $e_m=[\frac{mu}{\ell}[-]\frac{mv}{k}]$. 
\end{proof}

\begin{lemma}\label{lem2.7.2}
If $\lambda_i\not\in\Q$ for $i=0,1,\ldots,n-1,n,\infty$  then $\mathrm{gcd}(A,B)=1$ for all $m\in\N$.
\end{lemma}
\begin{proof}
By Lemma \ref{lem2.5.1}, 
\begin{align*}
A&=ma_{d-k}+\partial_x q_d=q_d\left(\sum_{i=1}^n(m\lambda_i-1)\frac{\ell c_ix^{\ell-1}}{y^k-c_ix^\ell}+\frac{1}{x}(u-\ell(\lambda_0+1)+1)\right),\\
B&=mb_{d-\ell}+\partial_y q_d=q_d\left(\sum_{i=1}^n(1-m\lambda_i)\frac{k y^{k-1}}{y^k-c_ix^\ell}+\frac{1}{y}(1-(v+k\lambda_\infty))\right).
\end{align*}
Suppose that $g=\mathrm{gcd}(A,B)$. Since $\lambda_i\not\in\Q$, we have $\mathrm{gcd}(A,x)=1$, $\mathrm{gcd}(B,y)=1$ and $\mathrm{gcd}(A,y^k-c_ix^\ell)=1$ for all $i=1,\ldots,n$. Therefore $\mathrm{gcd}(g,q_d)=1$. Moreover,
we have $g|q_d$ since $kxA+\ell yB=(d+m)q_d$. It implies that $\mathrm{gcd}(A,B)=1$.
\end{proof}

The following lemma will be used to normalize the hamiltonian part:
\begin{notation}
We say that a $1$-form $\omega_d\in\QQ^d(f)$ satisfies the \emph{generic condition} if $\lambda_i\not\in\Q$ for all  $i=0,1,\ldots,n-1,n,\infty$ and the coefficients of $A$ and $B$ satisfy the condition of non-vanishing determinant of matrix $M_m$ in \ref{2.2.7} for $m=1,\ldots,k\ell n-1$.
\end{notation}

\begin{lemma}\label{lem1.2.5}
Let $\omega_d\in\QQ^d(f)$ satisfy the generic condition and $\omega\in\QQ(\omega_d)$. Using the same notation as in Lemma \ref{lem2.7.1}, for each $m\geq 1$ there exist a diffeomorphism $\phi(x,y)=(x+x\alpha,y+y\beta)$ and a unity $u=1+\delta$ where $\alpha,\beta$, $\delta$ are quasi-homogeneous 
polynomials of degree $m$ such that $\tq_{d+m}=q_{d+m}+AU+BV=xy\tq'_{d+m}$ satisfies the conditions
\begin{itemize}
\item $deg_x\tq'_{d+m}\leq \ell n-1$ and $deg_y\tq'_{d+m}\leq kn-1
$  if $1\leq m\leq k\ell n-1$,
\item $\tq_{d+m}=0$ if $m\geq k\ell n$.  
\end{itemize}
\end{lemma}
\begin{proof}
Denote by $QP(m)$ the set of all $(k,\ell)$-quasi-homogeneous polynomials of degrees $m$. Then $QP(m)$ is a vector space of dimension $e_m$. Denote by $\bA=\frac{A}{y}$, $\bB=\frac{B}{x}$, $\bU=\frac{U}{x}$, $\bV=\frac{V}{y}$. By Lemma \ref{lem2.7.1}, 
$$\frac{\tq_{d+m}}{xy}=\frac{q_{d+m}}{xy}+\bA\bU+\bB\bV.$$

Consider the linear map
\begin{equation*}
 \Psi_m:QP(m)\times QP(m)\rightarrow QP(k\ell n+m)
\end{equation*}
\begin{equation*}
(\bU,\bV)\mapsto \bA \bU+\bB \bV.
\end{equation*}

\underline{Case $m \ge k\ell n$}. By Lemma \ref{lem2.7.2}, $\bA$ and $\bB$ are coprime. It implies that
\begin{equation}\label{eq13}
\Ker\Psi_m=\{(Z\bB,-Z\bA), Z\in QP(m-k\ell n)\}.
\end{equation}
Hence $\Psi_m$ is surjective due to the equality of dimensions of vector space
\begin{equation*}
e_{m}+e_{m}-e_{m+k\ell n}=e_{m-k\ell n}.
\end{equation*}
Consequently, there exists $\phi$ such that $\tq_{d+m}=AU + BV + q_{d+m}=0$.
 
\underline{Case $1\leq m\leq k\ell n-1$}. In this case $\mathrm{Ker}\Psi_m=\{0\}$. Denote by $NQP(k\ell n+m)$ the subspace of $QP(k\ell n+m)$ generalized by all the monomials $g(x,y)$ satisfying
\begin{equation*}
deg_xg\leq \ell n-1,\;deg_yg\leq k n-1.
\end{equation*}
We also denote by $NQP^\perp(k\ell n+m)$ the subspace of $QP(k\ell n+m)$ generalized by all the monomials $g(x,y)$ such that
\begin{equation*}
deg_xg\geq \ell n\;\text{or}\;deg_yg\geq k n.
\end{equation*}
Denote by $pr_m$ the standard projection
\begin{equation*}
pr_m:QP(k\ell n+m)\rightarrow NQP^\perp(k\ell n+m).
\end{equation*}
The proof is reduced to show that in a generic condition for all $q\in QP(k\ell n+m)$ there exists $\bA\bU+\bB \bV\in \Ima\Psi_m$ such that $$q+\bA\bU+\bB\bV\in NQP(k\ell n+m).$$
Since $$\mathrm{dim}NQP^\perp(k\ell n+m)=2e_m=\mathrm{dim}QP(m)\times QP(m),$$
this is equivalent to prove that in a generic condition $pr_m\circ\Psi_m$ is bijective.

Because $\bA,\bB\in QP(k\ell n)$, we can write $\bA=\sum_{i=0}^{n}A_ix^{\ell(n-i)}y^{ki}$, $\bB=\sum_{i=0}^{n}B_ix^{\ell i}y^{k(n-i)}$. Then the matrix representation $M_m$ of the linear map $pr_m\circ\Psi_m$ is given by
\begin{equation}\label{2.2.7}
M_m={\tiny{
\left[
\begin{array}{ccccc|ccccc}
A_0      & 0        & \ldots & 0      & 0 	  	   & B_n        & 0          & \ldots & 0		& 0        \\
A_1      & A_0      & \ldots & 0      & 0	  	   & B_{n-1}    & B_n        & \ldots & 0		 & 0 \\
\vdots   & \vdots   & \ddots & \vdots & \vdots 	   & \vdots     & \vdots     & \ddots & \vdots & \vdots	 \\
A_{e_m-2}& A_{e_m-3}& \ldots & A_0    & 0    	   & B_{n-e_m+2}& B_{n-e_m+3}& \ldots & B_n    &0      \\
A_{e_m-1}& A_{e_m-2}& \ldots & A_1    & A_0    	   & B_{n-e_m+1}& B_{n-e_m+2}& \ldots & B_{n-1}    & B_n  \\ \hline
A_n      & A_{n-1}  & \ldots & A_{n-e_m+2}& A_{n-e_m+1} & B_0        & B_{1}      & \ldots & B_{e_m-2} & B_{e_m-1}\\
0        & A_n      & \ldots & A_{n-e_m+3}& A_{n-e_m+2} & 0          & B_0        & \ldots & B_{e_m-3} & B_{e_m-2}\\
\vdots   & \vdots   & \ddots & \vdots     &\vdots 	   & \vdots     & \vdots     & \ddots & \vdots	& \vdots \\
0		 & 0        & \ldots & A_n          & A_{n-1}    	   & 0          & 0          & \ldots & B_0    & B_1\\
0		 & 0        & \ldots & 0          & A_{n}    	   & 0          & 0          & \ldots & 0    & B_0
\end{array}
\right]}}
\end{equation}
The determinant $\mathrm{det}M_m$ is a polynomial in $A_i$ and $B_j$. Since $\frac{\partial^{2e_m}\mathrm{det}M_m}{(\partial A_0^{e_m})(\partial B_0)^{e_m}}=(e_m!)^2\neq 0$, such polynomial is not identically zero. Therefore the condition $\mathrm{det}M_m\neq 0$ is satisfied for generic $\omega_d\in\QQ^d(f)$. 
\end{proof}
The following lemma will be used to normalize the radial part: 
\begin{lemma}\label{lem1.2.6}
If $\omega_d\in\QQ^d(f)$ satisfies $\lambda_i\not\in\Q$ for all $i=0,1,\ldots, n-1,n,\infty$, then there exist $\phi(x,y)=(x+x\alpha,y+y\beta)$ and $u(x,y)=1+\delta(x,y)$ where $\alpha,\beta$ and $\delta$ are quasi-homogeneous 
polynomials of degree $m$ such that $\Delta q_{d+m}=0$ and $\tb_{d+m-\ell}=b_{d+m-\ell}+\Delta b_{d+m-\ell}$ satisfies the following condition
\begin{equation*}
deg_y\tb_{d+m-\ell}\leq kn-1
\end{equation*}
\end{lemma}
\begin{proof}
Suppose that $\phi(x,y)=(x+x\alpha,y+y\beta)$ and $u(x,y)=1+\delta(x,y)$ where $\alpha,\beta$ and $\delta$ are quasi-homogeneous 
polynomials of degree $m$ such that $\Delta q_{d+m}=0$. By the proof of Lemma \ref{lem1.2.5}, we have $(\bU,\bV)\in \Ker\Psi_m$. It implies that $(\bU,\bV)=(Z\bB,-Z\bA)$ where $Z=0$ if $m\leq k\ell n$ and $Z\in QP(m-k\ell n)$ if $m>k\ell n$. Therefore, $W$ in Lemma \ref{lem2.7.1} can be written as follows:
\begin{eqnarray*}
W&=&\dy(a_{d-k}xZ\bB-b_{d-\ell}yZ\bA)-(\dy a_{d-k}-\dx b_{d-\ell})xZ\bB\\
&=& \dy(Z(a_{d-k}B-b_{d-\ell}A))-(\dy a_{d-k}-\dx b_{d-\ell})ZB.
\end{eqnarray*}
Denote by $C=\dy a_{d-k}-\dx b_{d-\ell}$. Then
\begin{align*}
a_{d-k}B-b_{d-\ell}A&=a_{d-k}(mb_{d-k}+\partial_y q_d)-b_{d-\ell}(ma_{d-k}+\partial_x q_d)\\
&=a_{d-k}\partial_y q_d-b_{d-\ell}\partial_x q_d\\
&=\frac{(q_d-\ell y b_{d-\ell})\partial_y q_d-kxb_{d-\ell}\partial_x q_d}{kx}\\
&=\frac{q_d\partial_y q_d-b_{d-\ell}(\ell y \partial_y q_d+kx\partial_x q_d)}{kx}\\
&=\frac{q_d\partial_y q_d-d b_{d-\ell}q_d}{kx}\\
&=\frac{q_d(\partial_y q_d-d b_{d-\ell})}{kx}\\
&=\frac{q_d(kx\partial_y a_{d-k}+ \ell y \partial_y b_{d-\ell}-(d-\ell)b_{d-\ell})}{kx}\\
&=\frac{q_d(kx\partial_y a_{d-k}- kx \partial_y b_{d-\ell})}{kx}\\
&=q_d(\partial_y a_{d-k}- \partial_y b_{d-\ell})\\&=q_dC.
\end{align*}
It follows
\begin{align*}
\Delta b_{d+m-l}&=\partial_y(Zq_dC)-CZB+\frac{mb_{d-\ell}\delta-q_d\partial_y\delta}{d+m}\\
&=(\partial_yq_d-B)ZC+q_d(\partial_y(ZC))+\frac{mb_{d-\ell}\delta-q_d\partial_y\delta}{d+m}\\
&=mb_{d-\ell}\left(\frac{\delta}{d+m}-ZC\right)-q_d\partial_y\left(\frac{\delta}{d+m}-ZC\right).
\end{align*}
Define the linear map
\begin{equation*}
 \Phi_m:QP(m)\rightarrow QP(k\ell n+m)
\end{equation*}
\begin{equation*}
T\mapsto m\frac{b_{d-\ell}}{x}T-\frac{q_d}{x}\dy T.
\end{equation*}
We will show that $\Phi_m$ is injective. Assume that there exists $T\in \Ker\Phi_m$ and $T\neq 0$. Decompose $T=r \left(\frac{q_d}{x}\right)^\gamma$ where $\gamma\in\N$ and $r$ does not divide $\frac{q_d}{x}$. It follows from $T\in\Ker\Phi_m$ that
\begin{equation}\label{2.2.8} r(mb_{d-\ell}-\gamma\dy q_d)=q_d\dy r.
\end{equation}
By Lemma \ref{lem2.5.1},
\begin{equation*}
mb_{d-\ell}-\gamma\dy q_d=-q_d\left(\sum_{i=1}^n(m\lambda_i+\gamma)\frac{ky^{k-1}}{y^k-c_ix^l}+\frac{m(k\lambda_\infty+v)+\gamma)}{y}\right).
\end{equation*}
If $\lambda_i\not\in\Q$ for all $i=1,\ldots,n,\infty$ then $\frac{1}{x}(mb_{d-\ell}-\gamma\dy q_d)$ and $\frac{q_d}{x}$ are coprime. It implies that $r\frac{(mb_{d-\ell}-\gamma\dy q_d)}{x}$ does not divide $\frac{q_d}{x}$ and this is contradictory to \eqref{2.2.8}.

Denote by $NQP_y(k\ell n+m)$ the subspace of $QP(k\ell n+m)$ generalized by the monomials  $g(x,y)$ such that
\begin{equation*}
deg_yg\leq kn-1.
\end{equation*}
Then for all $g\in NQP_y(k\ell n+m)$ we have 
\begin{equation}\label{2.2.9}
x^{[\frac{m}{k}[+1}|g(x,y).
\end{equation}
Denote by $NQP_y^\perp(k\ell n+m)$ the subspace of $QP(k\ell n+m)$ generalized by the monomials $g(x,y)$ such that
$$deg_yg\geq kn.$$
We also denote by $pr_{ym}$ the standard projection
\begin{equation*}
pr_{ym}:QP(k\ell n+m)\rightarrow NQP^\perp_y(k\ell n+m).
\end{equation*}
The proof is done if we can show that $pr_{ym}\circ\Phi_m$ is bijective from $QP(m)$ to $NQP_y^\perp(k\ell n+m)$.
Because  $\mathrm{dim}NQP_y^\perp(k\ell n+m)=e_m=\mathrm{dim}QP(m)$. It rests to show that $pr_{ym}\circ\Phi_m$ is injective. We claim that this is equivalent to prove that
\begin{equation}\label{1.2.10}
 \Ima\Phi_m\cap NPQ_y(k\ell n+m)=\{0\}.
\end{equation}
Indeed, if \eqref{1.2.10} is true, suppose that $g\in\Ker (pr_{ym}\circ\Phi_m)$. Then $\Phi_m(g)\in NPQ_y(k\ell n+m)$ which implies that $\Phi_m(g)=0$. It follows by the injectivity of $\Phi_m$ that $g=0$.

Now, let's prove \eqref{1.2.10}. Assume that there exists $T\neq 0$ such that $\Phi_m(T)\in NPQ_y(k\ell n+m)$. Decompose $T=x^\theta \bT$ where $x$ and $\bT$ are coprime. Because $T\in QP(m)$, 
we have $\theta\leq [\frac{m}{k}[$. By \eqref{2.2.9}, $x$ is a divisor of $m\frac{b_{d-\ell}}{x}\bT-\frac{q_d}{x}\dy\bT$. By Lemma \ref{lem2.5.1}, we have
\begin{equation}\label{eq22}
 m\frac{b_{d-\ell}}{x}(0,y)\bT(0,y)-\frac{q_d}{x}(0,y)\dy \bT(0,y)=
\end{equation}
\begin{equation*}
c_0y^{nk}\left( m(v-k(1+\lambda_0))\bT(0,y)-y\dy \bT(0,y)\right)
\end{equation*}
It is a contradiction because condition $\lambda_0\not\in\Q$ forces that the right hand side of \eqref{eq22} is different from $0$ for all non-zero polynomial $\bT(0,y)$.
\end{proof}
\begin{mythr}\label{theoremA}
For generic $\omega_d\in\QQ^d(f)$, each germ $\omega\in\QQ(w_d)$ is strictly formally orbitally equivalent to a unique form $\omega_{h,s}$
\begin{equation*}
\omega_{h,s}=\omega_d+d\left(xy h\right)+s(\ell ydx-kxdy), 
\end{equation*}
where 
\begin{equation*}
h(x,y)=\sum_{\substack{ki+\ell j\geq k\ell n+1\\ 0\leq i\leq \ell n-1\\ 0\leq j\leq kn-1}} h_{ij}x^iy^j,\; s(x,y)=\sum_{j=0}^{kn-1}s_j(x)x^{\ell n+1+]\frac{1-\ell j}{k}]}y^j,
\end{equation*} 
$s_i(x)$ are formal series on $x$.
\end{mythr}
\begin{proof}
By Corollary \ref{cor1.1.9}, we can decompose
$$\omega=\omega_d+d(xyh)+s(\ell ydx-kxdy),$$
where $h_i=\frac{q_{i+k+\ell}}{(i+k+\ell)xy}=\frac{q'_i}{i+k+\ell}$, $q'=\frac{q}{xy}$. Let's rewrite $s(x,y)$ as follows
\begin{align*}
s_{i}&=\frac{\dy a_{i+\ell}+\dx b_{i+k}}{i+k+\ell}\\
&=\frac{1}{i+k+\ell}\left(\frac{\dy(q_{i+k+\ell}-\ell y b_{i+k})}{kx}-\dx b_{i+k}\right)\\
&=\frac{\dy q_{i+k+\ell}}{(i+k+\ell)kx}-\frac{\ell b_{i+k}+\ell y\dy b_{i+k}+kx\dx b_{i+k}}{(i+k+\ell)kx}\\
&=\frac{\dy q_{i+k+\ell}}{(i+k+\ell)kx}-\frac{b_{i+k}}{kx}.
\end{align*} 
By Lemma \ref{lem1.2.5} and \ref{lem1.2.6}, we  can remove all the monomials $x^iy^j$, $i\geq \ell n$, $j\geq k n$ in the components of $q'$ and all the monomial $x^iy^j$, $j\geq kn$ in the components of $b$. This implies we can normalize $h$ and $s$ such that
\begin{equation*}
h(x,y)=\sum_{\substack{ki+\ell j\geq k\ell n+1\\ 0\leq i\leq \ell n-1\\ 0\leq j\leq kn-1}} h_{ij}x^iy^j,\; s(x,y)=\sum_{j=0}^{kn-1}s'_j(x)y^j,
\end{equation*} 
where $s'_j(x)$ are formal series of $x$. Because $s(x,y)$ only contains the monomials of degree at least $k\ell n+1$, for each $j=0,\ldots,kn-1$ $s'_j(x)$ divides $x^{i_j}$ where $i_j$ is the minimal integer such that $ki_j+\ell j\geq k\ell n+1$ and. Moreover, $ki_j+\ell j\geq k\ell n+1$ if and only if $i_j\geq \ell n+1+\big]\frac{1-\ell j}{k}\big]$. So we have
$$s'_j(x)=s_j(x)x^{\ell n+1+]\frac{1-\ell j}{k}]},$$
where $s'_j(x)$ are formal series of $x$. The uniqueness part is straightforward by the uniqueness in Lemma \ref{lem1.2.5} and \ref{lem1.2.6}.
\end{proof}

\begin{remark} Since every formal diffeomorphism $\phi$ can decompose as $\phi=\phi'\circ\phi_0$ where $\phi_0$ is a linear transformation and $\phi'$ is strict, the formal normal form for the case in which we do not require the strict condition can be easily obtained. The slightly difference is in the initial part. For the strict case we have $n$ free coefficients corresponding to the coordinates of the non-corner singularities in the principle component of divisor, but for the general case we can normalize one of them and let the others free.  
\end{remark}

\begin{example}
For $n=2,k=3$, $\ell=2$, the strict formal normal form is given by
$$\omega_d+d(xyh)+s(\ell ydx-kxdy),$$
where
\begin{align*}
\omega_d=c_0xy(y^3-c_1x^2)(y^3-c_2x^2)\Bigg(&\lambda_1\frac{d(y^3-c_1x^2)}{y^3-c_1x^2}+\lambda_2\frac{d(y^3-c_2x^2)}{y^3-c_2x^2}\\
&+(2\lambda_0+1)\frac{dx}{x}+(3\lambda_\infty+1)\frac{dy}{y}\Bigg),
\end{align*}
$$h(x,y)=h_{3,2}x^3y^2+h_{1,5}xy^5+h_{2,5}x^2y^4+ h_{3,3}x^3y^3+h_{2,5}x^2y^5+h_{3,4}x^3y^4+ h_{3,5}x^3y^5,$$ 
$$s(x,y)=s_0(x)x^5+s_1(x)x^4y+s_2(x)x^3y^2+ s_3(x)x^3y^3+ s_4(x)x^2y^4+s_5(x)xy^5.$$
In the formula of $\omega_d$ we have two free coefficients $c_1$, $c_2$ corresponding to the position of two singularities of the cuspidal branches. However, the formal normal form (non-strict) is given by 
$$\bar{\omega}_d+d(xyh')+s'(\ell ydx-kxdy),$$
where $h'$ and $s'$ have the same form as $h$ and $s$ respectively, but we can normalize one singularity, which has the coordinates  $(x_c,y_c)=(1,0)$, and so $\bar{\omega}_d$ has only one free coefficient $c$:
\begin{equation*}
\bar{\omega}_d=xy(y^3-x^2)(y^3-cx^2)\left(\lambda_1\frac{d(y^3-x^2)}{y^3-x^2}+\lambda_2\frac{d(y^3-cx^2)}{y^3-cx^2}+ (2\lambda_0+1)\frac{dx}{x}+(3\lambda_\infty+1)\frac{dy}{y}\right).
\end{equation*}
\end{example}

\begin{remark}
The number of free coefficients of $h$ in the normal form is consistent with the dimension of Mattei's moduli space. Indeed, for $m=1,\ldots,k\ell n-1$ we have  
\begin{align*}
\#\{(i,j)\in\N^2|ki+\ell j=k\ell n+m, i\leq \ell n-1, j\leq k n-1\}&=e_{k\ell n+m}-2e_m\\
&=n-e_m
\end{align*}
So the number of free coefficients of $h$ is given by
$$\delta'(\omega)=\sum_{m=1}^{k\ell n-1}(n-e_m)=n(k\ell n-1)-\sum_{m=1}^{k\ell n-1}e_m.$$
By \cite{Gen-Pau1}, the dimension of Mattei's moduli space is given by
\begin{align*}
\delta(\omega)&=\sum_{m=0}^{k\ell n-1}\left(]\frac{\ell-u}{\ell}(m-k\ell n)]-[\frac{k-v}{k}(m-k\ell n[\right)\\
&=\sum_{m=0}^{k\ell n-1}\left(n+]-\frac{mu}{\ell}]-[-\frac{mv}{k}[\right).
\end{align*}
We will show that for any real number $a$, $[a[+]-a]=-1$. Indeed, since $[a[\leq a<[a[+1$, $]-a]<-a\leq ]-a]+1$, we have  
$$[a[+]-a]<0<[a[+]-a]+2.$$
Therefore, $[a[+]-a]=-1$. It follows that
\begin{equation*}
\delta(\omega)=\sum_{m=0}^{k\ell n-1}\left(n-]\frac{mu}{\ell}]+[\frac{mv}{k}[\right)=n^2k\ell -\sum_{m=0}^{k\ell n-1}e_m.
\end{equation*}
The difference of $\delta(\omega)$ and $\delta'(\omega)$ is given by 
$$\delta(\omega)-\delta'(\omega)=n-e_0=n-1.$$
The reason for the existence of this difference is from the fact that we just consider the strict conjugation. The number $n-1$ then is corresponding to the number of free coefficients corresponding to the position of non-corner singularities in the non-strict formal normal form.
\end{remark}



\end{document}